\documentclass{amsart}
\usepackage{graphicx,amsfonts,amssymb,amsmath,amsthm}
\usepackage[dvipsnames]{xcolor}
\usepackage[pdftex]{hyperref} 
\hypersetup{colorlinks=true,linkcolor=blue,citecolor=blue,urlcolor=SkyBlue}
\usepackage{cite}
\usepackage[boxsize=1em]{ytableau}

\theoremstyle{plain}
\newtheorem{theorem}    {Theorem}[section] 
\newtheorem*{theorem*}    {Theorem}
\newtheorem{lemma}      [theorem]{Lemma}
\newtheorem{corollary}  [theorem]{Corollary}
\newtheorem{proposition}[theorem]{Proposition}

\theoremstyle{definition}

\theoremstyle{remark}
\newtheorem{remark}              {Remark}

\newtheorem*{notation*}            {Notation}

\newtheorem*{question*}   {Question}

\numberwithin{equation}{section}

\def\A{\mathbb A}
\def\C{\mathbb C}
\def\F{\mathbb F}

\def\R{\mathbb R}

\begin{document}

\title[Decomposition of symmetric powers for GL(3) and GL(4)]{On the conjectural decomposition of symmetric powers of automorphic representations for GL(3) and GL(4)}

\author{Nahid Walji}
\address{Department of Mathematics, University of British Columbia, Vancouver, B.C., V6T 1Z2, Canada}
\email{nwalji@math.ubc.ca}

\date{}
\begin{abstract} Given a cuspidal automorphic representation $\pi$ for GL(3) over a number field and a positive integer $k$, assume that the symmetric $m$th power lifts of $\pi$ are isobaric automorphic for $m \leq k$, cuspidal for $m \leq k-1$, and that certain associated Rankin--Selberg products are isobaric automorphic. Then the number of cuspidal isobaric summands in the $k$th symmetric power lift is bounded above by 3 when $k \geq 7$, and bounded above by 2 when $k \geq 19$ with $k \equiv 1 \pmod 3$. We then investigate the analogous problem for GL(4).
\end{abstract}
\maketitle

\section{Introduction}

Let $\pi$ be a cuspidal automorphic representation for GL($n$) over a number field $F$. Consider the incomplete automorphic $L$-function
$$L^X(s,\pi) = \prod_{v \not \in X}{\rm det}\left(I_n - A_v (\pi) {\rm N}v^{-s} \right)^{-1},$$
for ${\rm Re}(s)>1$, where $X$ is a finite set containing the archimedean places as well as the places at which $\pi$ is ramified, and $A_v(\pi)$ is the (Langlands) conjugacy class of $\pi$ at $v$, which can be represented by a diagonal matrix in ${\rm GL}_n(\C)$. Given the symmetric power map ${\rm Sym}^k: {\rm GL}_n(\C) \rightarrow {\rm GL}_{m}(\C)$, (where $m$ is dependent on $n$ and $k$) one can construct the $L$-function
$$L^{X}(s, \pi, {\rm Sym}^k) = \prod_{v \not \in X}{\rm det}\left(I_m - {\rm Sym}^k(A_v (\pi)) {\rm N}v^{-s} \right)^{-1}.$$
for $s$ in some right-half plane. The functoriality conjectures then predict the existence of an isobaric automorphic representation $\Pi$ for ${\rm GL}(m)/F$ such that 
\begin{align*}
L(s,\Pi) = L(s, \pi, {\rm Sym}^k).
\end{align*}
In that case we denote $\Pi$ by ${\rm Sym}^{k}(\pi)$. This automorphic representation is not necessarily cuspidal -- for example, in the case of $n = 2$ and $k = 2$, the symmetric square (which is known to be automorphic due to the work of Gelbart--Jacquet \cite{GJ78}) can either be cuspidal (which occurs precisely when $\pi$ is not an induced representation), have exactly two cuspidal isobaric summands, or three cuspidal isobaric summands. 

Let $\mathcal{N} (\Pi)$ denote the number of cuspidal isobaric summands of $\Pi$. 
Assuming that the $k$th symmetric power lift of $\pi$ is an isobaric automorphic representation, we then ask
\begin{question*}
What bound $B$ can we establish such that 
\begin{align*}
\mathcal{N}({\rm Sym}^{k}\pi) \leq B,
\end{align*}
for all cuspidal automorphic representations $\pi$ for ${\rm GL}(n)/F$?
\end{question*}

The trivial bound in this setting would be $B = m$. 
This bound is in fact sharp for certain cases, for example if $(n,k) = (2,2)$, then $B = 3$ is optimal. This is because when $\pi$ is induced from a Hecke character $\nu$ with respect to a quadratic extension $K/F$, if $\nu/ \nu ^ \tau$ is invariant under the non-trivial element $\tau$ of ${\rm Gal}(K / F)$ 
then the symmetric square lift is the isobaric sum of three Hecke characters.
Similarly, if $(n,k)= (2,4)$, then the trivial bound $B = 5$ is also sharp
(this occurs for induced automorphic representations with the properties described in the previous sentence).
In contrast, when $(n,k) = (2,3)$, the bound $B = 3$ (rather than the trivial bound of $B = 4$) is sharp. For example, when $\pi$ is induced from a Hecke character $\nu$ of order three, the symmetric cube is an isobaric automorphic representation whose cuspidal isobaric summands consist of a cuspidal automorphic representation for GL(2) and two Hecke characters (see \cite{KS02}).

In the case of $n = 2$ and $k \geq 5$, the automorphy of the $k$th symmetric power lift is not known in general. There are however some conditional results: 
Under the assumption that the symmetric fifth power lift is automorphic and that the symmetric square, cube, and fourth power lifts are cuspidal, Ramakrishnan \cite{ra09} proved that the symmetric fifth power lift is also cuspidal. Furthermore, fixing some $k \geq 6$, assuming automorphy for symmetric $m$th power lifts for $m \leq 2k$ and assuming that ${\rm Sym}^6 (\pi)$ is cuspidal, he proves that ${\rm Sym}^{k} (\pi)$ is cuspidal. 

The case of $n > 2$ seems to be less well understood. 
Given the ideas of \cite{ra09}, we consider what would occur in the GL(3) and GL(4) setting. 
The purpose of this paper is to determine what the Langlands functoriality conjectures imply about upper bounds on the number of cuspidal isobaric summands in symmetric power lifts.

\begin{theorem}\label{mt}
Let $\pi$ be a cuspidal automorphic representation for GL(3) over a number field $F$. Fix some integer $k \geq 2$. If the adjoint and  $m$th symmetric power lifts of $\pi$ for ${\rm max}(1,k-3) \leq m \leq  k$ are isobaric automorphic and furthermore cuspidal when ${\rm max}(1,k-3) \leq m \leq k-1$, and the Rankin--Selberg product of $\pi$ or $\widetilde{\pi}$ with any summand of ${\rm Sym}^k (\pi)$ is isobaric automorphic, then 
\begin{align*}
\mathcal{N}({\rm Sym}^{k}(\pi)) \leq
\begin{cases}
4, \text{ for }k \geq 2, \\
3, \text{ for }k \geq 7 \text{ or }k=3, 4, \\
2, \text{ for }k \geq 19 \text{ and } k \equiv 1 \bmod 3.
\end{cases}
\end{align*}
\end{theorem}

\begin{remark}
For $k = 2$, this bound is sharp; see Proposition \ref{propk2}.
\end{remark}
One can also ask what can be said about other functorial lifts under suitable automorphy assumptions. In the case of the adjoint lift of $\pi$, which is expected to be an automorphic representation for GL(8), under the assumption that both the adjoint and symmetric square lifts are automorphic, we show that the cuspidality of the adjoint lift is equivalent to the cuspidality of the symmetric square lift (see Corollary \ref{ads2}).
As for the case of exterior powers, it is known that if $\pi$ is cuspidal then so are $\Lambda^2 (\pi)$ and $\Lambda^3 (\pi)$, since $\Lambda^2 (\pi) = \widetilde{\pi}\otimes \omega$ and $\Lambda^3 (\pi) = \omega$, where $\omega$ denotes the central character of $\pi$. 

It is of interest to form a comparison by considering the case of GL(4). As one might expect, the bounds obtained are not as strong as in the GL(3) case. In particular, we have 
\begin{theorem}\label{mt2}
Let $\pi$ be a cuspidal automorphic representation for GL(4)$/F$. Fix some $k \geq 15$. Assume that the $m$th symmetric power lift of $\pi$ is isobaric automorphic for $k-4 \leq m \leq  k$, cuspidal for $m = k-3, k-1$, and that the Rankin--Selberg product of $\pi $ or $ \widetilde{\pi}$ with any summand of ${\rm Sym}^k (\pi)$ is isobaric automorphic, then 
\begin{align*}
\mathcal{N}({\rm Sym}^{k}(\pi)) \leq
\begin{cases}
6, \text{ for }k \geq 15,\\ 
5, \text{ for }k \geq 21,\\ 
4, \text{ for } k \geq 39, \\
3, \text{ for }k \geq 139 \text{ and } k \equiv 3,5,7 \bmod 8.
\end{cases}
\end{align*}
\end{theorem}

\begin{remark}
We also obtain a number of (weaker) bounds for the cases $3 \leq k \leq 14$. See Theorem \ref{mt2-p} and Propositions \ref{p-gl4k3} and \ref{p-gl4k4} for further details. The value of 3 appears to be the lowest bound possible given the method used; this is discussed at the end of Section \ref{pfmt} (see Remark \ref{expec}). 
\end{remark}

It should also be possible to produce analogous results for GL($n$), where $n \geq 5$, though we expect that the bounds obtained would become progressively weaker as $n$ grows large. 

This article is structured as follows. In Section \ref{bkgd}, we establish notation and cover some relevant background. In Section \ref{pfmt}, we establish some key identities and then prove Theorems \ref{mt} and \ref{mt2}. In Section \ref{aexs}, we provide some context for our bounds by considering some examples (under the strong Artin conjecture) of symmetric power decompositions.

\subsection*{Acknowledgements} The author would like to thank K. Martin for his helpful comments and in particular for identifying a number of groups with interesting symmetric power representations, adding cases to Section 4. The author also wishes to thank P. D. Nelson, D. Prasad, and A. Raghuram for their helpful feedback on an earlier draft of this paper.
The research of the author was supported in part by an NSERC Discovery Grant.

\section{Notation and Background} \label{bkgd}

\begin{notation*}
Let $\mathcal{A} ({\rm GL}_n (\A_F))$ denote the set of all automorphic representations for ${\rm GL}(n) /F$, and let the subset $\mathcal{A}_0 ({\rm GL}_n (\A_F))$ consist of those which are cuspidal. 
\end{notation*}

\subsection{Automorphic $L$-functions}\label{autlfn}
Let $\pi \in \mathcal{A} ({\rm GL}_n (\A_F))$. Its $L$-function then has the form 
\begin{align*}
L(s,\pi) = \prod_{v \in \Sigma_F} L_v (s, \pi)
\end{align*}
for ${\rm Re}(s)>1$, where $\Sigma_F$ is the set of all places of $F$. At a finite place $v$ at which $\pi$ is unramified, we have $L_v(s,\pi) = {\rm det} \left(I_n - A_v(\pi) {\rm N}v^{-s}\right)^{-1},$ where
the conjugacy class $A_v (\pi)$ can be represented by the matrix ${\rm diag}(\alpha(v,1), \dots, \alpha(v,n))$. The complex numbers $\alpha (v,j)$, $1 \leq j \leq n$, are known as the Satake parameters for $\pi$ at $v$.

\subsection{Symmetric power, exterior power, and adjoint $L$-functions}\label{symlfn}

Let ${\rm Sym}^k: {\rm GL}_n(\C) \rightarrow {\rm GL}_m(\C)$ be the $k$th symmetric power representation.
We can associate to $\pi$ its $k$th symmetric power $L$-function $$L(s, \pi , {\rm Sym}^k)=\prod_{v \in \Sigma_F} L_v (s, \pi,{\rm Sym}^k),$$ where $s$ lies in a suitable right-half plane. We have 
\begin{align*}
L_v(s, \pi , {\rm Sym}^k) = {\rm det} \left(I_n - {\rm Sym}^k (A_v(\pi)) {\rm N}v^{-s}\right)^{-1}.
\end{align*}
for any finite place $v$ at which $\pi$ is unramified.

We note that in the case of $n = 3$, the degree of the $k$th symmetric power $L$-function is $(k + 1)(k + 2)/2$, and for $n = 4$, the degree of the $k$th symmetric $L$-function is $(k + 1)(k + 2) (k + 3)/6$. As mentioned in the introduction, the Langlands functoriality conjectures predict that these $L$-functions are automorphic; that is, for each $k$ there exists an isobaric automorphic representation $\tau_k$ such that 
$L(s,\tau_k) = L(s, \pi , {\rm Sym}^k).$
When $n > 2$, this is not known for any $k \geq 2$.

In a similar manner, one can also define the $k$th exterior power $L$-functions and the adjoint $L$-function, where we replace ${\rm Sym}^k$ in the definition above by the $k$th exterior power representation, or the adjoint representation, of ${\rm GL}_n(\C)$. When $n = 3$, we note that the exterior square and cube lifts are both automorphic, since $L(s, \pi , \Lambda^2) = L(s, \widetilde{\pi}\otimes \omega_\pi)$ and $L(s, \pi , \Lambda^2) = L(s,\omega_\pi)$, where $\omega_\pi$ is the central character of $\pi$. 
When $n = 4$, the second, third, and fourth power exterior lifts are all automorphic. In particular, the degree 6 $L$-function $L(s, \pi , \Lambda^2)$ is known to be automorphic due to Kim \cite{Ki03} (who proved the correspondence for all local $L$-factors 
except those where $v \mid 2,3$ with $\pi_v$ supercuspidal) and Henniart \cite{He09} (who addressed the remaining cases).
Furthermore, $L(s, \pi , \Lambda^3) = L(s, \widetilde{\pi}\otimes \omega_\pi)$ and $L(s, \pi , \Lambda^4) = L(s,\omega_\pi)$.

Lastly, we note that the adjoint $L$-function associated to $\pi$ has degree $(n ^2 -1)$ and is not known to be automorphic for any $n > 2$.

\subsection{Rankin--Selberg $L$-functions}
Let $\pi \in \mathcal{A} ({\rm GL}_n (\A_F))$ and $\pi' \in \mathcal{A} ({\rm GL}_m (\A_F))$. 
We can then construct the Rankin--Selberg $L$-function 
\begin{align*}
L (s, \pi \times \pi') = \prod_{v \in \Sigma_F} L_v (s, \pi \times \pi'),
\end{align*}
where, for a finite place $v$ at which $\pi$ and $\pi'$ are unramified, we have 
\begin{align*}
L_v (s, \pi \times \pi') = {\rm det}\left(I_{nm} - A_v(\pi) \otimes A_v(\pi') Nv^{-s}  \right)^{-1} ,
\end{align*}
and the Euler product converges absolutely for ${\rm Re}(s)>1$.

Now assume that $\pi$ and $\pi'$ are cuspidal. If $\widetilde{\pi} \simeq \pi' \otimes |\cdot|^{it}$ for some $t \in \R$, then $L (s, \pi \times \pi')$ has simple poles at $s = it $ and $s= 1-it$ and is holomorphic everywhere else.
If not, then $L (s, \pi \times \pi')$ is entire and is invertible at $s=1$.

We say that $L (s, \pi \times \pi')$ is \textit{automorphic} if there exists an automorphic representation $\Pi$ (necessarily for ${\rm GL}(nm)/F$) such that 
\begin{align*}
L (s, \pi \times \pi') = L(s,\Pi).
\end{align*}

\subsection{Isobaric automorphic representations}

Given automorphic representations $\pi_1,$ $ \pi_2, \dots, \pi_k$ for ${\rm GL}(n_1), {\rm GL}(n_2), \dots, {\rm GL}(n_k)$, respectively, over $F$, there exists an automorphic representation $\Pi$ for ${\rm GL}(\sum_{j}n_j)/F$ such that 
\begin{align*}
L(s,\Pi) = \prod_{j}L(s,\pi_j).
\end{align*}
Then $\Pi$ is known as an isobaric automorphic representation and can be written as $\pi_1 \boxplus \pi_2 \boxplus \dots \boxplus \pi_k$. We will use the following notation for this paper: If $\pi_j$ is an isobaric summand of $\Pi$, we will write this as  $\pi_j \prec \Pi$.

\subsection{Schur polynomials}\label{schur}
The vector space of degree $d$ homogeneous symmetric polynomials in $k$ variables has a basis consisting of the Schur polynomials. We define these as follows:
Let $\lambda =( \lambda_1, \dots, \lambda_k)$ denote a partition of $d$, where $\lambda_1 \geq \dots \geq \lambda_k$. Then the \textit{Schur polynomial} with respect to $\lambda$, in variables $x_1, \dots, x_k$, is defined as
\begin{align*}
S_\lambda = \frac{|x_j ^{\lambda_i + k-i}|}{\Delta},
\end{align*}
where $\Delta = \prod_{i < j}(x_i - x_j)$ and $|a_{ij}|$ denotes the determinant of the matrix $(a_{ij})$.

For example, for $k = 3$, here are some of the polynomials that arise in the proofs of Lemmas \ref{prop1} and \ref{prop2}:
\begin{align*}
S_{(1)}&= x+y+z\\
S_{(1,1)}&= xy+xz+yz\\
S_{(1,1,1)}&= xyz\\
S_{(2)} &= 
x^2 + y^2 + z^2 +xy +xz +yz\\
S_{(2,1)} &= x^2 y + x y^2 + x^2 z + y^2 z + x z^2 + y z^2 + 2 x y z \\
S_{(2,2)}&= x^2 y^2 + x^2 z^2  + y^2 z^2 + x^2 y z + x y^2 z + x y z^2. 
\end{align*}

Let $\lambda =( \lambda_1, \dots, \lambda_k)$ and $\mu =( \mu_1, \dots, \mu_k)$ be partitions of numbers $d_1$ and $d_2$, respectively. Given the Littlewood--Richardson rule \cite{LR13,Sc77}, the product of the Schur polynomials $S_\lambda \cdot S_\mu$ can be expressed as $\sum_{\nu} C_{\lambda, \mu, \nu } S_\nu$ where the coefficients $C_{\lambda, \mu, \nu }$ are equal to the number of ways in which the Young diagram for $\lambda$ can be extended strictly by $\mu$ to obtain the Young diagram for the partition $\nu$.

\subsection{The strong Artin conjecture}

Given some number field $F$, consider a representation $\rho: {\rm Gal}(\overline{F} / F) \rightarrow {\rm GL}_n(\C)$. The strong Artin conjecture\cite{La70} implies that there exists a corresponding automorphic representation $\pi$ such that their $L$-functions coincide 
\begin{align*}
L(s,\rho) = L(s,\pi).
\end{align*}

The strong Artin conjecture is known in the case $n = 1$. In the case $n = 2$, this is known for all representations with projective image isomorphic to a dihedral group, tetrahedral group \cite{La80}, octahedral group \cite{Tu81}, or icosahedral group provided the representation is odd (a representation is defined to be odd if ${\rm det}\rho (c) = -1$, where $c$ represents complex conjugation) and $F$ is a totally real field \cite{KW09a,KW09b, Ki09, PS16, Sa19}.
When $n = 3$, this is known when $\rho$ is an induced representation \cite{jpss81} or the twist of a symmetric square representation \cite{GJ78}.
 For $n = 4$, the strong Artin conjecture has been proved for various cases where $\rho$ is solvable \cite{Ra01,Ma03,Ma04}. Various other cases have also been addressed (see \cite{AC90,Wo17,Wo18}).

\section{Proof of theorems} \label{pfmt}

We begin by establishing some identities that will be used in the proofs of the theorems.
\begin{lemma}\label{prop1}
Let $\pi \in \mathcal{A}_0 ({\rm GL}_3 (\A_F))$.
Let ${\rm std}$ denote the standard representation of ${\rm GL}_3(\C)$ and  ${\rm Sym}^{0}$ the trivial representation.
For any integer $m \geq 3$, we have, for $s$ in some suitable right-half plane,
\begin{align}\label{id}
& L^X(s, \pi, {\rm Sym}^{m-1} \times {\rm std}) L^X(s, \pi, {\rm Sym}^{m-3} \times \Lambda^3) \\
=& L^X(s, \pi, {\rm Sym}^{m}) L^X(s,\pi, {\rm Sym}^{m-2} \times \Lambda^2 ) \notag 
\end{align}
where $X$ is the set containing exactly the finite places at which $\pi$ is ramified and all the archimedean places.
\end{lemma}

\begin{proof}
Using Young diagrams, we obtain the following identities of Schur polynomials in $k$ variables (see Section \ref{schur} for notation) 
\begin{align*}
S_{(m-2)} \cdot S_{(1,1)} &= S_{(m-1,1)} + S_{(m-2,1,1)}\\
S_{(m-1)} \cdot S_{(1)} &= S_{(m)} + S_{(m-1,1)}\\
S_{(m-3)} \cdot S_{(1,1,1)} &= S_{(m-2,1,1)} + S_{(m-3,1,1,1)}.
\end{align*}
For example, the first equation can be obtained by taking Young diagram associated to the partition $(1,1)$ and extending it by the Young diagram associated to the partition $(m-2)$,
 \ytableausetup
 {mathmode, boxsize=1em}
which gives  
\begin{align*}
\begin{ytableau}
\bullet &  & & & \none[...]
&  &  & \\
\bullet
\end{ytableau}\text{ \quad  and  \quad }
\begin{ytableau}
\bullet &  &  & & \none[...]
&  &  \\
\bullet \\
*(white)
\end{ytableau}
\end{align*}
which correspond to the partitions $(m-1,1)$ and $(m-2,1,1)$, respectively.

Setting $k = 3$, the polynomial $S_{(m-3,1,1,1)}$ vanishes and we have that 
\begin{align}\label{eqnd}
S_{(m-1)} \cdot S_{(1)} + S_{(m-3)} \cdot S_{(1,1,1)} = S_{(m)} + S_{(m-2)} \cdot S_{(1,1)}.
\end{align}

For $v \not \in X$, denote the Satake parameters associated to $\pi_v$ as $\{\alpha_v,\beta_v,\gamma_v\}$. 
Then the multiset of monomials, with multiplicity, in $S_{(\ell)}(\alpha_v,\beta_v,\gamma_v)$ is equal to the multiset of eigenvalues of ${\rm Sym}^\ell (A_v(\pi))$ (using the notation introduced in Sections \ref{autlfn} and \ref{symlfn}). The same statement holds true for the pair $S_{(1,1)} (\alpha_v,\beta_v,\gamma_v)$ and $\Lambda^2 (A_v(\pi))$, as well as the pair  $S_{(1,1,1)} (\alpha_v,\beta_v,\gamma_v)$ and $\Lambda^3 (A_v(\pi))$.

We use these in conjunction with identity (\ref{eqnd}) to establish the local analogue of equation (\ref{id}) at each $v \not \in X$, thereby proving the lemma.
\end{proof}

\begin{lemma}\label{prop2}
For $\pi \in \mathcal{A}_0 ({\rm GL}_n (\A_F))$, we have 
\begin{align}\label{prop2id}
L^X(s,\pi, {\rm Sym}^2  \times \widetilde{{\rm Sym}^2 }) L^X(s, \pi, {\rm Sym}^2  \times \widetilde{\Lambda^2}) L^X(s, \pi, \widetilde{{\rm Sym}^2 }\times \Lambda^2 ) L^X(s, \pi, \Lambda^2 \times \widetilde{\Lambda^2} ) \notag \\
= L^X(s,\pi, {\rm Ad} \times {\rm Ad}) L^X(s, \pi, {\rm Ad})^2 L^X(s, 1),
\end{align}
for $s$ in some right-half plane.
\end{lemma}

\begin{proof}
This is presumably well-known to the experts
(for example, it is conceptually similar to the identities mentioned in \cite{PR21}), 
and follows from the Schur polynomial identity
 \begin{align*}
 & S_{(2)} \cdot S_{(2,2)}  +  S_{(2)}\cdot S_{(1)} \cdot S_{(1,1,1)}  +  S_{(2,2)}\cdot S_{(1,1)}  + S_{(1,1)}\cdot S_{(1)} \cdot S_{(1,1,1)}\\
 =&  S_{(2,1)}\cdot S_{(2,1)}  +  2 \cdot S_{(2,1)}\cdot S_{(1,1,1)} + S_{(1,1,1)}^2.\notag 
 \end{align*}
\end{proof}

\begin{corollary}\label{ads2}
Fix $n = 3$ and assume that the symmetric square and adjoint lifts of $\pi$ are automorphic. Then ${\rm Sym}^2 (\pi)$ is cuspidal if and only if ${\rm Ad} (\pi)$ is cuspidal.
\end{corollary}
\begin{proof}
If ${\rm Sym}^2 (\pi)$ is not cuspidal, then 
\begin{align*}
- {\rm ord}_{s = 1}L(s, {\rm Sym}^2 (\pi) \times \widetilde{{\rm Sym}^2 (\pi)}) \geq 2,
\end{align*}
and so the left-hand side of equation (\ref{prop2id}) has a pole of order at least 3 at $s=1$. This implies
\begin{align*}
- {\rm ord}_{s = 1}L(s, {\rm Ad} (\pi) \times {\rm Ad} (\pi)) \geq 2,
\end{align*}
which means that the adjoint lift cannot be cuspidal. We proceed similarly in the other direction. 
\end{proof}

\begin{lemma}\label{p-id2}
If $\pi \in \mathcal{A}_0({\rm GL}_4(\A_F))$, then for any integer $m \geq 4$ we have 
\begin{align}\label{id2}
& L^X(s, \pi,{\rm Sym}^{m}) L^X(s, \pi, {\rm Sym}^{m-2} \times \Lambda^2)L^X(s, {\rm Sym}^{m-4}\times \Lambda^4) \\
=& L^X(s, \pi, {\rm Sym}^{m-1} \times  {\rm std}) L^X(s,\pi, {\rm Sym}^{m-3} \times \Lambda^3) \notag 
\end{align}
for $s$ in some right-half plane.
\end{lemma}

\begin{proof}
Using Young diagrams, we establish the following identities of Schur polynomials (in $k$ variables)
\begin{align*}
S_{(1,1)}\cdot S_{(m-2)} &= S_{(m-1,1)}+S_{(m-2,1,1)}\\
S_{(1,1,1,1)}\cdot S_{(m-4)} &= S_{(m-3,1,1,1)}+ S_{(m-4,1,1,1,1)} \\
S_{(m-1)}\cdot S_{(1)} &= S_{(m)} + S_{(m-1,1)} \\
S_{(1,1,1)} \cdot S_{(m-3)} &= S_{(m-2,1,1)} + S_{(m-3,1,1,1)}.
\end{align*}
Setting $k = 4$, the polynomial $S_{(m-4,1,1,1,1)}$ in the second equation vanishes and we obtain 
\begin{align*}
S_{(m)} + S_{(1,1)}\cdot S_{(m-2)} + S_{(1,1,1,1)}\cdot S_{(m-4)} = S_{(m-1)}\cdot S_{(1)} + S_{(1,1,1)} \cdot S_{(m-3)}.
\end{align*}
This allows us to establish the local analogue of identity (\ref{id2}) for all $v \not \in X$, and so proving the lemma.
\end{proof}

\begin{lemma}\label{lem}
Let  $A,B,C,$ and $D$ be automorphic representations for ${\rm GL}(a),$ ${\rm GL}(b),$ $ {\rm GL}(c) ,$ and $ {\rm GL}(d)$, respectively, all over $F$. Let $S$ be any finite set of places of $F$. Assume that the Rankin--Selberg product of $B$ and $\widetilde{D}$ is automorphic. If  $C$ is cuspidal and 
\begin{align*}
L^S(s, A)L^S(s, B)= L^S(s, C \times D)
\end{align*}
then we must have $ c \leq b \cdot d$.
\end{lemma}
\begin{proof}
From 
\begin{align*}
L^S(s, A)L^S(s, B)= L^S(s, C \times D)
\end{align*}
we have 
\begin{align*}
L^S(s, A \times \widetilde{B})L^S(s, B \times \widetilde{B})= L^S(s, C \times D \boxtimes \widetilde{B}).
\end{align*}
Since $-{\rm ord}_{s=1}L^S(s, B \times \widetilde{B}) \geq 1$ and
$-{\rm ord}_{s=1}L^S(s, A \times \widetilde{B}) \geq 0$, we have that
$L^S(s, C \times D \boxtimes \widetilde{B}) \geq 1$. Therefore $\widetilde{C} \prec D \boxtimes \widetilde{B}$ (recall from Section \ref{bkgd} that this notation means $\widetilde{C}$ is an isobaric summand of $D \boxtimes \widetilde{B}$), implying the inequality.
\end{proof}

The case of $k \geq 4$ in Theorem \ref{mt} will follow from:
 \begin{theorem}\label{mt-p}
Fix $k \geq 4$. Consider $\pi \in \mathcal{A}_0({\rm GL}_3(\A_F))$ such that ${\rm Sym}^m (\pi)$ is automorphic for $k-3 \leq m \leq k$, cuspidal for $m=k-3, k-1$, and the Rankin--Selberg product of $\pi $ or $ \widetilde{\pi}$ with any summand of ${\rm Sym}^k (\pi)$ is automorphic. Then 
 \begin{align*}
 \mathcal{N}({\rm Sym}^{k}(\pi)) \leq  \left\lfloor \frac{(k + 1)(k + 2)/2}{\left\lceil {k (k+1)}/{6}\right\rceil} \right\rfloor ,
 \end{align*}
 and if $k \equiv 1 \bmod 3$, we have a stronger bound of
 \begin{align*}
 \mathcal{N}({\rm Sym}^{k}(\pi)) \leq \left\lfloor \frac{(k + 1)(k + 2) /2}{\left\lceil {3k (k+1)}/{16}\right\rceil} \right\rfloor .
 \end{align*}
 \end{theorem}
\begin{proof}[Proof of Theorem \ref{mt-p}]
Assume that ${\rm Sym}^{k}(\pi)$ is not cuspidal. Then it has a cuspidal
isobaric summand $\tau$ for some GL($r$) such that 
$$r \leq \left\lfloor \frac{(k+1)(k+2)}{4} \right\rfloor.$$ So
\begin{align*}
-{\rm ord}_{s=1}L^X(s, {\rm Sym}^{k}(\pi) \times \widetilde{\tau})\geq 1.
\end{align*}
Given our assumption on the automorphy of the Rankin--Selberg product of $\pi$ with summands of ${\rm Sym}^k (\pi)$, we also have 
\begin{align*}
-{\rm ord}_{s=1}L^X(s, {\rm Sym}^{k-2}(\pi) \times \Lambda^2 (\pi) \boxtimes \widetilde{\tau})\geq 0,
\end{align*}
due to the non-vanishing of Rankin--Selberg $L$-functions at $s=1$.
This implies, via equation (\ref{id}),
that 
\begin{align*}
-{\rm ord}_{s=1}L^X(s, {\rm Sym}^{k-1}(\pi) \times \pi \boxtimes \widetilde{\tau})  -{\rm ord}_{s=1}L^X(s, {\rm Sym}^{k-3}(\pi) \otimes \omega \times \widetilde{\tau})\geq 1,
\end{align*}
where $\omega$ is the central character of $\pi$.

On the one hand, if 
\begin{align} \label{lfeq1}
-{\rm ord}_{s=1}L^X(s, {\rm Sym}^{k-3}(\pi) \otimes \omega \times \widetilde{\tau})\geq 1,
\end{align}
then
$r = (k-2)(k-1)/2.$

On the other hand, if
\begin{align*}
-{\rm ord}_{s=1}L^X(s, {\rm Sym}^{k-3}(\pi) \otimes \omega \times \widetilde{\tau})=0,
\end{align*}
then 
\begin{align}\label{lfeq2}
-{\rm ord}_{s=1}L^X(s, {\rm Sym}^{k-1}(\pi) \times \pi \boxtimes \widetilde{\tau}) \geq 1
\end{align}
and so ${\rm Sym}^{k-1}(\pi) \prec \widetilde{\pi}\boxtimes \tau$. Therefore 
$k (k + 1)/2 \leq 3  r$ and so
\begin{align}\label{eqrbd}
r \geq \left\lceil \frac{k (k+1)}{6}\right\rceil.
\end{align}
We get
\begin{align}\label{eqbd}
 \mathcal{N}({\rm Sym}^{k}(\pi)) \leq  \left\lfloor \frac{(k + 1)(k + 2)/2}{\left\lceil {k (k+1)}/{6}\right\rceil} \right\rfloor ,
\end{align}
for $k \geq 5$.

When $k = 4$ and equation (\ref{lfeq1}) holds, we address the case of $r = (k-2)(k-1)/2$.
In this case $r = 3$ and so $\tau = \pi \otimes \omega$. Therefore,
\begin{align*}
&-{\rm ord}_{s=1}L^X(s, {\rm Sym}^{k}(\pi) \times \widetilde{\tau})\geq 1 \text{\ \  and }\\
&-{\rm ord}_{s=1}L^X(s, {\rm Sym}^{k-2}(\pi) \times \Lambda^2 (\pi) \boxtimes \widetilde{\tau}) \geq 1,
\end{align*}
where for the latter equation we have used that
\begin{align*}
L^X(s, {\rm Sym}^{2}(\pi) \times \Lambda^2 (\pi) \boxtimes \widetilde{\tau}) 
= L^X(s, {\rm Sym}^2 (\pi) \times \widetilde{{\rm Sym}^2 (\pi)}) L^X(s, {\rm Sym}^2 (\pi) \times \widetilde{\Lambda^2 (\pi)}).
\end{align*}
Given that
\begin{align*}
-{\rm ord}_{s=1}L^X(s, {\rm Sym}^{k-1}(\pi) \times \pi \boxtimes \widetilde{\tau})&=0 \text{\ \  and }\\
-{\rm ord}_{s=1}L^X(s, {\rm Sym}^{k-3}(\pi) \otimes \omega \times \widetilde{\tau})&=1,
\end{align*}
we have a contradiction, so equations (\ref{lfeq2}) and therefore (\ref{eqbd}) hold.

\ 

If equation (\ref{lfeq2}) holds and $k \equiv 1 \pmod 3$, then we can improve on the bound in (\ref{eqbd}) for $k \geq 19$. From (\ref{lfeq2}), we have
\begin{align*}
{\rm Sym}^{k-1}(\pi) \boxplus \eta =  \widetilde{\pi}\boxtimes \tau
\end{align*}
for some automorphic representation $\eta$ for ${\rm GL}(3r- k(k + 1)/2 )$ over $F$. Given Lemma \ref{lem}, we obtain
\begin{align*}
r &\leq 3 \left(3r - \frac{k (k + 1)}{2}\right) \\
r &\geq \left\lceil \frac{3}{16} k (k + 1) \right\rceil.
\end{align*}
Therefore,
\begin{align*}
 \mathcal{N}({\rm Sym}^{k}(\pi)) \leq  \left\lfloor \frac{(k + 1)(k + 2)/2}{\left\lceil {3k (k+1)}/{16}\right\rceil} \right\rfloor .
\end{align*}
\end{proof}

The following two propositions will address the cases of $k=2$ and $k=3$, which are not covered by Theorem \ref{mt-p}.

\begin{proposition}\label{propk2}
Given $\pi \in \mathcal{A}_0({\rm GL}_3(\A_F))$, assume that its symmetric square lift is isobaric automorphic. Then $$\mathcal{N}({\rm Sym}^2 (\pi)) \leq 4.$$ Furthermore, this bound is sharp.
\end{proposition}

\begin{proof}
We will bound the number of Hecke characters that are isobaric summands of ${\rm Sym}^2 (\pi)$.
If there are none, then $\mathcal{N}({\rm Sym}^2 (\pi)) \leq 3$. Therefore, let us assume that there exists a Hecke character $\omega$ that is a summand of ${\rm Sym}^2 (\pi)$.
Given that 
\begin{align}\label{sym2id}
L^X(s, \pi \times \pi \otimes \omega) = L^X(s, {\rm Sym}^2 (\pi) \otimes \omega) L^X(s, \Lambda^2 (\pi) \otimes \omega),
\end{align}
the first $L$-function must have a pole at $s=1$, and so $\widetilde{\pi} = {\pi \otimes \omega}$. We define $T:= \{\omega \in \mathcal{A}({\rm GL}_1(\A_F)): \pi \otimes \omega = \widetilde{\pi}\}$.

In our setting, we know that  $\pi$ can only admit self-twists by at most nine different Hecke characters (including the trivial character). Indeed,
if $\pi$ admits a self-twist $\omega$, then 
\begin{align*}
-{\rm ord}_{s=1}L^X(s, (\pi \otimes \omega) \times \widetilde{\pi}) = 1,
\end{align*}
 so, given the automorphy of $\pi \boxtimes \widetilde{\pi}$ (since it is equal to the automorphic representation $({\rm Sym}^2 (\pi) \boxplus \Lambda^2 (\pi)) \otimes \omega$), it would mean that $\omega$ is a summand of $\pi \boxtimes \widetilde{\pi}$, thus there are only at most nine such possible characters.

Let $S:=\{\mu \in  \mathcal{A}({\rm GL}_1(\A_F)): \pi \otimes \mu = \pi\}$ be the multiplicative group of self-twists of $\pi$. Since $S$ acts (by multiplication) faithfully and transitively on the set $T$, they have the same cardinality. Since the non-trivial self-twists must have order three, the size of $S$ is either 1, 3, 5, 7, or 9.
Given equation (\ref{sym2id}), $T$ cannot be of size 5, 7, or 9. If $|T|= 1$ or $3$, then $\mathcal{N}({\rm Sym}^2 (\pi)) \leq 4.$

To establish that this bound is sharp, we first note that the group  $A_4$  has a 3-dimensional irreducible representation whose symmetric square decomposes into three 1-dimensional representations and one 3-dimensional irreducible representation.
An example of a 3-dimensional irreducible Artin representation with image isomorphic to $A_4$ can be found in \cite{lmfdb} under the label \href{https://www.lmfdb.org/ArtinRepresentation/3.3136.4t4.a.a}{3.3136.4t4.a.a}.
The strong Artin conjecture is known in this case \cite{AC90},
and so this representation corresponds to a cuspidal automorphic representation $\pi$ for GL(3) where $\mathcal{N}({\rm Sym}^2 (\pi)) = 4$.
\end{proof}

\begin{proposition}\label{propk3}
For $\pi \in \mathcal{A}_0({\rm GL}_3(\A_F))$, assume that its symmetric square lift is automorphic and furthermore cuspidal, and that its adjoint and symmetric cube lifts are isobaric automorphic. Then 
\begin{align*}
\mathcal{N}({\rm Sym}^3 (\pi)) \leq 3.
\end{align*}
\end{proposition}

\begin{proof}
We follow the beginning of the proof of theorem \ref{mt-p} to establish that either equation (\ref{lfeq1}) or (\ref{lfeq2}) holds.

If equation (\ref{lfeq1}) is true, then $\tau = \omega$. However, this implies that 
\begin{align*}
-{\rm ord}_{s=1}L^X(s, \pi \times \Lambda^2 (\pi) \boxtimes \widetilde{\tau})= 1,
\end{align*}
and so, using (\ref{id}) with $m = 3$, we have 
\begin{align*}
-{\rm ord}_{s=1}L^X(s, {\rm Sym}^{2}(\pi) \times \pi \boxtimes \widetilde{\tau})  -{\rm ord}_{s=1}L^X(s,  \omega \times \widetilde{\tau})\geq 2.
\end{align*}
This is a contradiction since 
\begin{align*}
-{\rm ord}_{s=1}L^X(s, {\rm Sym}^{2}(\pi) \times \pi \boxtimes \widetilde{\tau})&=0, \text{ and }\\
-{\rm ord}_{s=1}L^X(s, \omega \times \widetilde{\tau})&=1.
\end{align*}
This proves that the symmetric cube cannot have any Hecke characters as isobaric summands. 
Therefore equation (\ref{lfeq2}) is true.

Now we show that the symmetric cube cannot have more than one cuspidal automorphic representation for GL(2) occurring as a summand. 
Let us say that two distinct such summands $\tau,\tau'$ exist. Then given (\ref{lfeq2}) we must have 
\begin{align*}
\widetilde{{\rm Sym}^2 (\pi)} = \pi \boxtimes \widetilde{\tau} = \pi \boxtimes \widetilde{\tau'}.
\end{align*}
 Therefore, 
\begin{align*}
-{\rm ord}_{s=1}L^X(s,  \pi \boxtimes \widetilde{\tau} \times \widetilde{\pi} \boxtimes {\tau'}) = 1
\end{align*}
and so 
\begin{align*}
-{\rm ord}_{s=1}L^X(s,  ({\rm Ad}(\pi) \boxplus 1) \times \widetilde{\tau} \boxtimes {\tau'}) = 1,
\end{align*}
using the known automorphy results for ${\rm GL}(2) \times {\rm GL}(2)$ \cite{ra00} and ${\rm GL}(2) \times {\rm GL}(3)$ \cite{KS00}.
By Corollary \ref{ads2}, ${\rm Ad} (\pi)$ is cuspidal, which implies that $\widetilde{\tau} \boxtimes {\tau'}$ must contain the trivial character, and so $\tau = \tau'$.

Furthermore, a given GL(2) summand $\tau$ cannot occur more than once in ${\rm Sym}^3 (\pi)$, since otherwise
\begin{align*}
-{\rm ord}_{s=1}L^X(s, {\rm Sym}^{3}(\pi) \times \widetilde{\tau})\geq 2,
\end{align*}
whereas 
\begin{align*}
-{\rm ord}_{s=1}L^X(s, {\rm Sym}^{2}(\pi) \times \pi \boxtimes \widetilde{\tau})  -{\rm ord}_{s=1}L^X(s,  \omega \times \widetilde{\tau})= 1,
\end{align*}
leading to a contradiction.

Since (under the assumptions of the proposition) ${\rm Sym}^3 (\pi)$ is an automorphic representation for GL(10), it therefore cannot have more than 3 isobaric summands.
\end{proof}

\begin{remark}
The proof of the proposition above says that $r \geq 2$ when $k = 3$. Under the strong Artin conjecture, this bound is sharp due to examples from Sections \ref{sec648} and \ref{ex216} of representations whose symmetric cube decomposes into a 2-dimensional and an 8-dimensional representation.
\end{remark}

We now prove a more detailed version of Theorem \ref{mt2}, namely
\begin{theorem}\label{mt2-p}
Fix some $k \geq 4$. Assume that $\pi \in \mathcal{A}_0({\rm GL}_4(\A_F))$ has ${\rm Sym}^m (\pi)$  isobaric automorphic for all $k-4 \leq m \leq k$, cuspidal for $m = k-3, k-1$, and that the Rankin--Selberg product of $\pi $ or $\widetilde{\pi}$ with any summand of ${\rm Sym}^k (\pi)$ is isobaric automorphic. Then 
 \begin{align*}
 \mathcal{N}({\rm Sym}^{k}(\pi)) \leq  \left\lfloor \frac{(k + 1)(k + 2) (k + 3)/6}{\left\lceil {(k-2) (k-1)k}/{24}\right\rceil} \right\rfloor ,
 \end{align*}
 and if $k \equiv 3,5,\text{or } 7 \bmod 8$, then 
 \begin{align*}
 \mathcal{N}({\rm Sym}^{k}(\pi)) \leq \left\lfloor \frac{(k + 1)(k + 2) (k + 3)/6}{\left\lceil {2(k-2) (k-1)k}/{45}\right\rceil} \right\rfloor .
 \end{align*}
\end{theorem}

\begin{proof}
This follows in a similar way to the proof of Theorem \ref{mt-p}. We sketch the proof below.

We assume that ${\rm Sym}^k (\pi)$ is not cuspidal, and so it has a cuspidal summand $\tau$ for GL(r). 
Using equation (\ref{id2}), we find that either 
\begin{align*}
L^X(s, {\rm Sym}^{k-1}(\pi) \times  \pi \boxtimes \widetilde{\tau}) 
\end{align*}
or 
\begin{align*}
L^X(s,{\rm Sym}^{k-3}(\pi)  \times \Lambda^3 (\pi) \boxtimes \widetilde{\tau}) 
\end{align*}
has a pole at $s=1$.

In the first case, ${\rm Sym}^{k-1}(\pi) \prec \widetilde{\pi}\boxtimes \tau$, and so
$k (k + 1)(k + 2)/6 \leq 4 r$, implying
\begin{align*}
r \geq \left\lceil \frac{k (k + 1)(k + 2)}{24}\right\rceil.
\end{align*}

In the second case, ${\rm Sym}^{k-3}(\pi) \prec (\pi \otimes \widetilde{\omega}) \boxtimes \tau$, and so we have
\begin{align*}
r \geq \left\lceil \frac{(k-2) (k-1)k}{24}\right\rceil.
\end{align*}
We conclude
 \begin{align*}
 \mathcal{N}({\rm Sym}^{k}(\pi)) \leq  \left\lfloor \frac{(k + 1)(k + 2) (k + 3)/6}{\left\lceil {(k-2) (k-1)k}/{24}\right\rceil} \right\rfloor .
 \end{align*}

We improve this bound for $k \equiv 3,5,\text{or } 7 \bmod 8$. If ${\rm Sym}^{k-3}(\pi) \prec (\pi \otimes \widetilde{\omega}) \boxtimes \tau$, we then can write 
 \begin{align} \label{eqn2}
 {\rm Sym}^{k-3}(\pi) \boxplus \eta = (\pi \otimes \widetilde{\omega}) \boxtimes \tau
 \end{align}
 for an automorphic representation $\eta$ for ${\rm GL}(4r - (k-2) (k-1)k/6)$ over $F$.
 Applying lemma \ref{lem} to equation (\ref{eqn2}) above, we obtain
 \begin{align*}
 r \geq \left\lceil {2(k-2) (k-1)k}/{45}\right\rceil,
 \end{align*}
 and we conclude that 
\begin{align*}
\mathcal{N}({\rm Sym}^{k}(\pi)) \leq \left\lfloor \frac{(k + 1)(k + 2) (k + 3)/6}{\left\lceil {2(k-2) (k-1)k}/{45}\right\rceil} \right\rfloor .
\end{align*}
\end{proof}

One can also obtain further bounds in the cases of $k = 3$ and $4$:
\begin{proposition}\label{p-gl4k3}
Given $\pi \in \mathcal{A}_0({\rm GL}_4(\A_F))$ such that its symmetric square and cube lifts are isobaric automorphic, ${\rm Sym}^2 (\pi)$ is cuspidal, and the Rankin--Selberg product of $\pi $ or $\widetilde{\pi}$ with any summand of ${\rm Sym}^3 (\pi)$ is isobaric automorphic, then 
\begin{align*}
\mathcal{N}({\rm Sym}^{3}(\pi)) \leq 6.
\end{align*}
\end{proposition}

\begin{proof}
In a similar manner to the proof of Lemma \ref{p-id2}, we obtain 
\begin{align*}
 L^X(s, \pi,{\rm Sym}^3) L^X(s, \pi, {\rm std} \times \Lambda^2) 
= L^X(s, \pi, {\rm Sym}^2 \times  {\rm std}) L^X(s,\pi, \Lambda^3) 
\end{align*}
for $s$ in some suitable right-half plane. Following the proof of Theorem \ref{mt2-p} and using the assumptions of this proposition, we find that either 
\begin{align*}
- {\rm ord}L^X(s, {\rm Sym}^2 (\pi) \times \pi \boxtimes \widetilde{\tau})  &\geq 1 \text{\ \ or }\\
- {\rm ord}L^X(s, \Lambda^3 (\pi) \times \widetilde{\tau})  &\geq 1,
\end{align*}
where $\tau$ is an automorphic representation for GL(r) that is an isobaric summand of ${\rm Sym}^3 (\pi)$. If the former inequality is true, then $r \geq 3$; if the latter inequality holds, then $r \geq 4$.
Therefore, $\mathcal{N}({\rm Sym}^{3}(\pi)) \leq 6.$
\end{proof}

\begin{proposition} \label{p-gl4k4}
Given $\pi \in \mathcal{A}_0({\rm GL}_4(\A_F))$ such that ${\rm Sym}^m (\pi)$ is isobaric automorphic for $m \leq 4$, cuspidal for $m \leq 3$, and that the adjoint lift and the products of $\pi, \widetilde{\pi}$ with any summand of ${\rm Sym}^4 (\pi)$ are isobaric automorphic. 
If $\pi$ is neither essentially self-dual, nor admits a (non-trivial) self-twist, then
\begin{align*}
\mathcal{N}({\rm Sym}^{4}(\pi)) \leq 7.
\end{align*}
\end{proposition}

\begin{proof}
From \cite{Ki03} and \cite{AR11}, we know that the exterior square lift of $\pi$ is automorphic cuspidal. Then by Lemma \ref{prop2}, we have that ${\rm Ad}(\pi)$ is cuspidal.
Following the proof of Theorem \ref{mt2-p}, we note that either
\begin{align*}
&L^X(s, {\rm Sym}^3(\pi) \times  \pi \boxtimes \widetilde{\tau}) 
\text{\ \ or }\\
&L^X(s, ({\rm Ad} (\pi)\omega  \boxplus \omega ) \times \widetilde{\tau}) 
\end{align*}
has a pole at $s=1$. If the first $L$-function has a pole at $s=1$, we continue to proceed as in the proof of Theorem \ref{mt2-p}.
If the second $L$-function has a pole at $s=1$, then either $\tau = {\rm Ad} (\pi)$ or $\omega$. Given Lemma \ref{p-id2}, $\omega$ can only occur as a summand of ${\rm Sym}^4 (\pi)$ at most once.
We conclude that $\mathcal{N}({\rm Sym}^{4}(\pi)) \leq 7.$

\end{proof}

\begin{remark} \label{expec}
For Theorem \ref{mt-p}, as $k \rightarrow \infty$ the right-hand side of the first displayed inequality tends to 3 and the right-hand side of the second displayed inequality tends to 2. In the case of Theorem \ref{mt2-p}, as $k \rightarrow \infty$, the right-hand side of the first displayed inequality tends to 4 and the right-hand side of the second displayed inequality tends to 3.
Given the results of \cite{ra09}, it may be reasonable to expect that for sufficiently large $k$, under the conditions of Theorems \ref{mt-p} and \ref{mt2-p}, the upper bound on $\mathcal{N}({\rm Sym}^k \pi)$ should be 1.
\end{remark}

\section{Examples (via the strong Artin conjecture)} \label{aexs}

To provide some context for our bounds in Theorems \ref{mt} and \ref{mt2} (most of which are not expected to be sharp), we consider decompositions of symmetric power lifts of representations for certain finite groups. Under the strong Artin conjecture,
these would correspond to decompositions of symmetric power lifts of various automorphic representations for GL(3) and GL(4). Sections \ref{ex1080}--\ref{ex432} concern 3-dimensional irreducible representations and sections \ref{exsl29}--\ref{ex1440} consider 4-dimensional irreducible representations.

In this section, we will repeatedly make use of results from \cite{bl17,Ma04,MW17}.
Calculations were carried out using GAP \cite{GG}, and some groups below are referred to using GAP notation, where the group $[x,y]$ corresponds to \textit{SmallGroup(x,y)} in GAP. We are also grateful to K. Martin for identifying a number of finite groups of interest (using Magma \cite{mag}), which include those from sections \ref{ex640} and \ref{ex1440}.

The first two subsections contain some detail of the representations and their decompositions, whereas for the remaining subsections the details are not included but follow in a similar way.

\subsection{3-dimensional representations}

\subsubsection{The Valentiner group $V_{1080}$} \label{ex1080}
This group is the unique nonsplit central extension of $A_5$ by $C_3$ and is known as \textit{SmallGroup(1080,260)} in GAP. It has 17 irreducible representations, four of which are three-dimensional and faithful. We list some characters of interest to us in the table below, where we have preserved the conjugacy class names and character numbering as used in GAP. We have

\begin{center}
\scalebox{0.8}{
\begin{tabular}{p{0.4cm}|c|c|c|c|c|c|c|c|c|c|c|c|c|c|c|c|c}
& 1a& 3a& 3b& 2a&  6a&  6b& 3c& 3d& 12a& 12b& 4a& 5a& 15a& 15b& 15c& 5b& 15d\\
\hline
$\chi_2$      &3&  $-3\alpha$& $-3\overline{\alpha}$& $-1$&   $\alpha$&  $\overline{\alpha}$&  . & . & $-\alpha$& $-\overline{\alpha}$&  1&  $\beta$&   $\gamma$&  $\overline{\gamma}$&  $\overline{\delta}$& $\beta^*$&   $\delta$\\
$\chi_3$      &3&  $-3\alpha$& $-3\overline{\alpha}$& $-1$&   $\alpha$&  $\overline{\alpha}$&  . & . & $-\alpha$& $-\overline{\alpha}$&  1& $\beta^*$&   $\delta$&  $\overline{\delta}$&  $\overline{\gamma}$&  $\beta$&   $\gamma$\\
$\chi_4 $     &3& $-3\overline{\alpha}$&  $-3\alpha$& $-1$&  $\overline{\alpha}$&   $\alpha$&  . & . &$-\overline{\alpha}$&  $-\alpha$&  1&  $\beta$&  $\overline{\gamma}$&   $\gamma$&   $\delta$& $\beta^*$&  $\overline{\delta}$\\
$\chi_5  $    &3& $-3\overline{\alpha}$&  $-3\alpha$& $-1$&  $\overline{\alpha}$&   $\alpha$&  . & . &$-\overline{\alpha}$&  $-\alpha$&  1& $\beta^*$&  $\overline{\delta}$&   $\delta$&   $\gamma$&  $\beta$&  $\overline{\gamma}$\\
$\chi_8   $   &6&  $-6 \alpha$& $-6 \overline{\alpha}$&  2&   $-2\alpha$&  $-2\overline{\alpha}$&  . & . &  . &  . & . & 1&  $-\alpha$& $-\overline{\alpha}$& $-\overline{\alpha}$&  1&  $-\alpha$\\
$\chi_9    $  &6& $-6 \overline{\alpha}$&  $-6 \alpha$&  2&  $-2\overline{\alpha}$&   $-2\alpha$&  . & . &  . &  . & . & 1& $-\overline{\alpha}$&  $-\alpha$&  $-\alpha$&  1& $-\overline{\alpha}$\\
$\chi_{12}$     & 9 & 9 & 9 & 1 &  1 &  1 & . & . &  1 &  1 & 1 &$-1$ &  $-1$ &  $-1$ &  $-1$ & $-1$ &  $-1$ \\
$\chi_{13} $    &9  &$-9\alpha$ &$-9\overline{\alpha}$  & 1 & $-\alpha$& $-\overline{\alpha}$ & .&  .&  $-\alpha$& $-\overline{\alpha}$ & 1 &$-1$ &   $\alpha$ & $\overline{\alpha}$  &$\overline{\alpha}$ &$-1$ &   $\alpha$ \\
$\chi _{14}$    &9 &$-9\overline{\alpha}$  &$-9\alpha$ & 1 &$-\overline{\alpha}$ & $-\alpha$&  . & .& $-\overline{\alpha}$  &$-\alpha$&  1 &$-1$   &$\overline{\alpha}$  & $\alpha$ &   $\alpha$ & $-1$  &$\overline{\alpha}$\\
$\chi_{15}$   &10& 10& 10& $-2$&  $-2$&  $-2$&  1&  1&   . &  . & . & . &  . &  . &  . & . &  .\\
$\chi_{16}$   &15&  $-15 \alpha$& $-15 \overline{\alpha}$& $-1$&   $\alpha$&  $\overline{\alpha}$&  . & . &  $\alpha$&  $\overline{\alpha}$& $-1$&  . &  . &  . &  . & . &  .\\
$\chi_{17}$   &15& $-15 \overline{\alpha}$&  $-15 \alpha$& $-1$&  $\overline{\alpha}$&   $\alpha$&  . & . & $\overline{\alpha}$&   $\alpha$& $-1$&  . &  . &  . &  . & . &  .
\end{tabular}
}
\end{center}
where the dots indicate zeros and
\begin{align*}
\alpha &= \frac{1 + \sqrt{-3}}{2},\\
\beta &= (1-\sqrt{5})/2, \quad \beta^* = (1+\sqrt{5})/2, \\
\gamma &=  - e^{2\pi i 7/15}-e^{2\pi i 13/15}= - \alpha \beta,\\
\delta &= -e^{2\pi i/15}-e^{2\pi i 4/15}= -\alpha \beta^* .
\end{align*}

The four irreducible representations of dimension 3 all have corresponding symmetric square and symmetric cube representations that are irreducible. In particular,
${\rm Sym}^2 (\chi_2) = {\rm Sym}^2 (\chi_3) =\chi_9 \text{ and }{\rm Sym}^2 (\chi_4) = {\rm Sym}^2 (\chi_5) = \chi_8$.
 Furthermore, $ {\rm Sym}^3 (\chi_j) = \chi_{15},$ for $j = 2,3,4,$ and $5$.

The symmetric fourth power lifts are not irreducible. They decompose into the sum of the symmetric square representation with an irreducible nine-dimensional representation,
\begin{align*}
{\rm Sym}^4 (\chi_2) = {\rm Sym}^4 (\chi_3) = \chi_8 + \chi_{13} \\
{\rm Sym}^4 (\chi_4) = {\rm Sym}^4 (\chi_5) = \chi_9 + \chi_{14} ,
\end{align*}
and so $$\mathcal{N}\left({\rm Sym}^4 (\chi_j)\right) = 2,$$ for $j = 2,3,4,$ and $5$.

\subsubsection{The groups [648,531], [648,532], and [648,533]}\label{sec648}

These groups are three of the four nonsplit central extensions of $G_{216}$ by $C_3$, where $G_{216}$ is the group $[216,153]$ in GAP and an explicit matrix presentation is provided in \cite{Ma04}. The fourth nonsplit central extension is not included as it does not have any 3-dimensional faithful irreducible representations.

For the first group, we list a few relevant lines of the character table, including the six 3-dimensional faithful irreducible representations of the group. We again preserve the conjugacy class notation and character numbering from GAP. Due to the width of the character table, we split it in two.

\begin{center}
\scalebox{0.8}{
\begin{tabular} {r|c|c|c|c|c|c|c|c|c|c|c|c|}
&1a & 3a & 3b & 3c & 2a &  6a &  6b & 4a & 12a & 12b &  9a & 9b \\ \hline 
$\chi_{4}$ & 2 &  2 &  2 &  2 & $-2$ &  $-2$ &  $-2$ &  . &  . &  . & $-1$ &  $-1$  \\ 
$\chi_{8}$ &  3 &  $-3\alpha$ &$-3\overline{\alpha}$  & . &$-1$ &   $\alpha$ & $\overline{\alpha}$  & 1 &  $-\alpha$ &$-\overline{\alpha}$  &  $\gamma$ &  $\delta$   \\ 
$\chi_{9}$ &  3 &  $-3\alpha$ &$-3\overline{\alpha}$  & . &$-1$ &   $\alpha$ & $\overline{\alpha}$  & 1 &  $-\alpha$ &$-\overline{\alpha}$  &  $\delta$ &  $\varepsilon$  \\ 
$\chi_{10}$ & 3 &  $-3\alpha$ &$-3\overline{\alpha}$  & . &$-1$ &   $\alpha$ & $\overline{\alpha}$  & 1 &  $-\alpha$ &$-\overline{\alpha}$  &  $\varepsilon$ &  $\gamma$  \\
$\chi_{11}$ & 3 & $-3\overline{\alpha}$  & $-3\alpha$ & . &$-1$ &  $\overline{\alpha}$  &  $\alpha$ & 1 & $-\overline{\alpha}$  & $-\alpha$ & $\overline{\delta }$ & $\overline{\varepsilon}$  \\
$\chi_{12}$ & 3 & $-3\overline{\alpha}$  & $-3\alpha$ & . &$-1$ &  $\overline{\alpha}$  &  $\alpha$ & 1 & $-\overline{\alpha}$  & $-\alpha$ & $\overline{\gamma }$ & $\overline{\delta }$  \\
$\chi_{13}$ & 3 & $-3\overline{\alpha}$  & $-3\alpha$ & . &$-1$ &  $\overline{\alpha}$  &  $\alpha$ & 1 & $-\overline{\alpha}$  & $-\alpha$ & $\overline{\varepsilon}$ & $\overline{\gamma }$  \\
$\chi_{19}$ & 6 & $-6\overline{\alpha}$  & $-6\alpha$ & . & 2 &  $-2\overline{\alpha}$  &  $-2\alpha$ & . &  . &  . &$-\overline{\gamma }$ &$-\overline{\delta }$  \\
$\chi_{22}$ & 8 &  8 &  8 & $-1$ &  . &  . &  . & . &  . &  . & $-2\overline{\alpha}$  & $-2\overline{\alpha}$    \\
\end{tabular}}
\end{center}
\begin{center}
\scalebox{0.8}{
\begin{tabular} {r|c|c|c|c|c|c|c|c|c|c|c|c}
& 9c&9d&18a & 18b & 18c & 9e & 9f & 9g& 9h &18d& 18e& 18f\\ \hline
$\chi_{4}$ & $-1$ &  $-1$& 1 &   1 &   1 &  $-1$ &  $-1$ &  $-1$& $-1$ & 1& 1 & 1\\
$\chi_{8}$ &  $\varepsilon$ &  .&  $-\beta^7$ &  $-\beta^4$ &  $-\beta$ & $\overline{\delta }$ & $\overline{\gamma }$ & $\overline{\varepsilon}$& .  &$-\beta^5$& $-\beta^2$& $-\beta^8$\\
$\chi_{9}$ &  $\gamma$ &  .&  $-\beta^4$ &  $-\beta$ &  $-\beta^7$ & $\overline{\varepsilon}$ & $\overline{\delta }$ & $\overline{\gamma }$&  .  &$-\beta^8$& $-\beta^5$&  $-\beta^2$\\
$\chi_{10}$ &  $\delta$ &  .&  $-\beta$ &  $-\beta^7$ &  $-\beta^4$ & $\overline{\gamma }$ & $\overline{\varepsilon}$ & $\overline{\delta }$& .  & $-\beta^2$& $-\beta^8$& $-\beta^5$\\
$\chi_{11}$ & $\overline{\gamma }$ &  .& $-\beta^5$ & $-\beta^8$ & $-\beta^2$ &  $\varepsilon$ &  $\delta$ &  $\gamma$ & .  & $-\beta$&  $-\beta^4$&  $-\beta^7$\\
$\chi_{12}$ & $\overline{\varepsilon}$ &  .& $-\beta^2$ & $-\beta^5$ & $-\beta^8$ &  $\delta$ &  $\gamma$ &  $\varepsilon$ & .   & $-\beta^4$&  $-\beta^7$&  $-\beta$\\
$\chi_{13}$ & $\overline{\delta }$ &  .& $-\beta^8$ & $-\beta^2$ & $-\beta^5$ &  $\gamma$ &  $\varepsilon$ &  $\delta$ & .  & $-\beta^7$  &$-\beta$  &$-\beta^4$\\
$\chi_{19}$ &$-\overline{\varepsilon}$ &  .& $-\beta^2$ & $-\beta^5$ & $-\beta^8$ & $-\delta$ & $-\gamma$ & $-\varepsilon$ &  .  & $-\beta^4$&  $-\beta^7$& $-\beta$\\
$\chi_{22}$ & $-2\overline{\alpha}$  & $\overline{\alpha}$&  . &  . &  . &  $-2\alpha$ &  $-2\alpha$ &  $-2\alpha$& $\alpha$&  . &  . &  .
\end{tabular}}
\end{center}

where, using the notation $e (\cdot) = e ^{2\pi i (\cdot)}$,
\begin{align*}
\alpha &= \frac{1 + \sqrt{-3}}{2},\\
\beta &= e(1/9),\\
\gamma &= 2e(4/9)+e(7/9), \\
\delta &= -e(4/9)-2e(7/9), \\
\varepsilon &=-e(4/9)+e(7/9).
\end{align*}

We first consider the symmetric power lifts of the character $\chi_8$, obtaining 
\begin{align*}
{\rm Sym}^2 (\chi_8) &= \chi_{19}\\
{\rm Sym}^3 (\chi_8) &= \chi_4 + \chi_{22}.
\end{align*}
Similarly, the characters $\chi_9,  \dots, \chi_{13}$ are all twists or twisted duals of the character $\chi_8$ and therefore exhibit the same decomposition pattern -- their symmetric square lifts are irreducible and their symmetric cube lifts decompose into two irreducible characters of dimension 2 and 8. Therefore, 
$$\mathcal{N}\left({\rm Sym}^3 (\chi_j)\right) = 2,$$
for $j = 8, \dots, 13$.

The other two groups [648,532] and [648,533] also have six faithful 3-dimensional irreducible representations whose symmetric squares are irreducible and whose symmetric cubes decompose into an irreducible 2-dimensional representation and an irreducible 8-dimensional representation.

\subsubsection{${\rm PSL}_2(\F_7)$}
There are exactly two 3-dimensional faithful irreducible representations of this group. Their corresponding symmetric square representations are irreducible (and equal to each other). The symmetric cube representation in both cases is not irreducible but rather the direct sum of a 3-dimensional representation and a 7-dimensional representation. So we have $$\mathcal{N}({\rm Sym}^3 (\rho)) = 2,$$
where $\rho$ denotes either 3-dimensional irreducible representation.

\subsubsection{The groups [216,88] and [432,239]} \label{ex216} \label{ex432}
Both groups have eight 3-dimensional faithful irreducible representations. Each representation is twist-equivalent to three others, and is a twisted dual of the remaining four. They all have irreducible symmetric square lifts, and their symmetric cubes each decompose into an 8-dimensional and a 2-dimensional irreducible representation, so
\begin{align*}
\mathcal{N}({\rm Sym}^3 (\rho))= 2,
\end{align*}
with $\rho$ denoting any of the 3-dimensional faithful irreducible representations.

\subsubsection{$S_4$ and $A_4$}
The group $S_4$  has two faithful 3-dimensional irreducible representations $\rho_1,\rho_2$, which are twist-equivalent by a quadratic character. Their symmetric square decomposes into a direct sum of 1-dimensional, 2-dimensional, and 3-dimensional irreducible representations, so 
\begin{align*}
\mathcal{N}({\rm Sym}^2 (\rho_j))= 3.
\end{align*}

\label{a4ex}
The group $A_4$ has one faithful 3-dimensional irreducible representation $\rho$.
Its symmetric square decomposes into three 1-dimensional representations and one 3-dimensional representation, so
\begin{align*}
\mathcal{N}({\rm Sym}^2 (\rho))= 4.\\
\end{align*}

\subsection{4-dimensional representations}
\subsubsection{${\rm SL}(2,\F_9)$} \label{exsl29}
This group (referred to as \textit{SmallGroup(720,409)} in GAP) has two faithful irreducible representations that are 4-dimensional.

These representations have the same symmetric square lift, namely the unique self-dual 10-dimensional irreducible representation of ${\rm SL}(2,\F_9)$. Their symmetric cubes decompose into two 10-dimensional irreducible representations, so
\begin{align*}
\mathcal{N}({\rm Sym}^3 (\rho))= 2,
\end{align*}
where $\rho$ denotes either 4-dimensional irreducible representation.

\subsubsection{The groups [640,21454] and [640,21455]}\label{ex640}
The first group
 has four faithful 4-dimensional irreducible representations. Each has an irreducible symmetric square lift and a symmetric cube lift that decomposes into irreducible representations of dimension 4 and 16. The same statements hold true for the group [640,21455]. So
$$\mathcal{N}({\rm Sym}^3 (\rho))= 2,$$
where $\rho$ is any irreducible 4-dimensional representation of [640,21454] or [640,21455].

\subsubsection{The group [1440,4591]} \label{ex1440}
This group has four faithful 4-dimensional irreducible representations. Each one has irreducible symmetric square and symmetric cube lifts. Their symmetric fourth power lifts all decompose into two 5-dimensional, one 9-dimensional, and one 16-dimensional irreducible representation. Therefore 
\begin{align*}
\mathcal{N}({\rm Sym}^4 (\rho))= 4,
\end{align*}
where $\rho$ denotes any of the 4-dimensional irreducible representations.

\end{document}